%% file: main.tex
\input{En_tete.tex}

\title{A finite-graph conjecture related to $\theta(p_c)=0$ for
Bernoulli bond percolation on $\mathbb Z^d$}
\author{Lucas Flammant\thanks{
Email: \texttt{lucasflammant@aol.com}. 
Alternative email: \texttt{lucasflammant6@gmail.com}.
}}
\date{\today}

\begin{document}

\maketitle

\input{Debut.tex}

\input{Preuve_planaire.tex}

\input{Reflexions_conjecture.tex}

\input{Appendice.tex}

\input{Acknowledgements.tex}

\bibliographystyle{plain}
\bibliography{biblio.bib}

\end{document}

%% file: En_tete.tex
\documentclass[11pt,a4paper]{scrartcl}

\usepackage[T1]{fontenc}
\usepackage[utf8]{inputenc} % Facultatif avec un LaTeX récent
\usepackage{microtype}

\usepackage{amsmath}
\usepackage{amssymb}
\usepackage{amsthm}
\usepackage{stmaryrd}

\usepackage{graphicx}
\usepackage{xcolor}
\usepackage{float}
\usepackage{needspace}
\usepackage{placeins}

\usepackage{tikz}
\usepackage{tikz-qtree}
\usetikzlibrary{
  fit,
  calc,
  positioning,
  decorations.pathreplacing,
  matrix,
  trees,
  patterns
}

\usepackage{geometry}
\usepackage[numbers]{natbib}
\usepackage{hyperref}
\hypersetup{hidelinks}

\numberwithin{equation}{section}

\theoremstyle{plain}
\newtheorem{theorem}{Theorem}[section]

\newtheorem{lemma}[theorem]{Lemma}

\newtheorem{proposition}[theorem]{Proposition}
\newtheorem{conj}[theorem]{Conjecture}

\theoremstyle{definition}
\newtheorem{definition}[theorem]{Definition}

\newtheorem{notation}[theorem]{Notation}

\theoremstyle{remark}

\newtheorem{remark}[theorem]{Remark}

\newcommand{\Z}{\mathbb{Z}}
\newcommand{\N}{\mathbb{N}}
\newcommand{\R}{\mathbb{R}}
\newcommand{\ori}{o}
\newcommand{\bo}{B}
\newcommand{\eps}{\varepsilon}
\newcommand{\vph}{\varphi}
\newcommand{\pr}{\mathbb{P}}
\newcommand{\ex}{\mathbb{E}}
\newcommand{\conn}{\longleftrightarrow}

\newcommand{\ind}{\mathbf{1}}

\newcommand{\blast}{\mathbf{b^*}}
\newcommand{\bfirst}{\mathbf{b}_\mathrm{first}}

\makeindex

%% file: Debut.tex
\begin{abstract}
We introduce a conjecture for Bernoulli bond percolation on finite graphs.
Roughly speaking, it asserts that if each boundary vertex is associated
with a highly probable event that is increasing with respect to the
percolation configuration, then, conditionally on the origin being
connected to the boundary, the origin is likely to be connected to a
boundary vertex whose associated event occurs. The conjecture implies
a high-probability connectivity statement which is known to imply
$\theta(p_c)=0$ for Bernoulli bond percolation on $\mathbb Z^d$, for
every $d\geq2$. We prove the conjecture for finite planar graphs when
the origin and the boundary vertices lie on the outer-face boundary,
with the explicit bound $1-2\sqrt{\varepsilon}$. The proof combines a
left-first depth-first exploration with a partial FKG inequality adapted
to monotonicity up to a stopping time. A counterexample shows that the
connectivity structure is essential: the analogous statement fails when
the connectivity events are replaced by arbitrary increasing events.
\end{abstract}

\medskip


\medskip


\section{Introduction}

Bernoulli bond percolation is one of the central models of probability
theory and statistical physics. Given a graph $G=(V,E)$ and a parameter
$p\in[0,1]$, each edge is declared open independently with probability
$p$ and closed otherwise. For background on percolation, we refer
to~\cite{GrimmettPercolation}.

This paper introduces a new conjecture concerning finite percolation
models and establishes it in a planar setting. The conjecture is
formulated in terms of a distinguished vertex, called the origin, and
a distinguished subset of vertices, referred to as the boundary.
Roughly speaking, it asserts that if one associates to every boundary
vertex a highly probable event that is increasing with respect to the
percolation configuration, then the open cluster of the origin is very
unlikely to avoid all boundary vertices whose associated event occurs.

The motivation for this conjecture comes from critical percolation.
The conjecture has the following consequence: it implies
$\theta(p_c)=0$ for Bernoulli bond percolation on $\mathbb Z^d$,
that is, critical Bernoulli bond percolation on $\mathbb Z^d$ has no
infinite cluster almost surely, for every $d\geq2$.

After the first version of this paper was completed, we became aware
of the related work of Kozma and Nitzan~\cite{kozma2024reductionthetapc0}.
They formulated a high-probability connectivity conjecture for percolation
on graphs and proved that it implies $\theta(p_c)=0$ on $\mathbb Z^d$.
Our conjecture implies the finite constant-parameter version of their
Conjecture~3, from which the general version follows by standard
approximation arguments. Their Theorem~6 then yields $\theta(p_c)=0$
on $\mathbb Z^d$; see Theorem~\ref{thm_theta_pc}. We nevertheless
intend to include our own proof in a subsequent version. Although the
two arguments are close at the conceptual level, their implementations
differ in several respects, and including our argument will also make
the paper self-contained.

While the general conjecture remains open, we prove it here for finite
planar graphs under a geometric assumption on the origin and the
boundary, with the explicit bound $1-2\sqrt{\varepsilon}$. The proof
combines a left-first depth-first exploration with geometric properties
specific to the planar setting and a partial version of the classical
FKG inequality~\cite{Harris1960,FKG1971}, adapted to monotonicity up to
a stopping time.

The specific form of the connectivity events is essential. Indeed, we
give a counterexample showing that the analogous statement fails if
these events are replaced by arbitrary increasing events.

The remainder of the paper is organized as follows. Section~2 introduces
the framework, states the general conjecture and the planar theorem, and
derives its implication for critical percolation on $\mathbb Z^d$.
Section~3 is devoted to the proof of Theorem~\ref{planar_thm}.
Section~4 contains further remarks on the general conjecture and
discusses possible extensions of the planar argument. The appendix
contains the proof of the partial FKG proposition.

\paragraph{Use of artificial intelligence.}
Artificial intelligence tools were used during the preparation of this
paper as an aid for mathematical discussion, for checking some arguments,
and for improving the exposition and wording of the manuscript. However, it was not used to generate the mathematical content: the mathematical content of the paper, including the statements, proofs, and mathematical developments, is the author's own work.

\section{General framework, conjecture and main results}

Throughout the paper, let $G=(V,E)$ be a finite undirected loopless
multigraph. Parallel edges are allowed and are regarded as distinct
elements of $E$. We distinguish a vertex $\ori\in V$, called the \emph{origin}, and a subset of vertices $\bo \subseteq V$, called the \emph{boundary}. Let $p \in (0,1)$ be a fixed parameter.

Let $(\Omega,\mathcal F,\mathbb P)$ be a probability space. We consider a random percolation configuration
\begin{equation*}
\pi=(\pi_e)_{e\in E}\in\{0,1\}^E,
\end{equation*}
where the random variables $(\pi_e)_{e\in E}$ are independent Bernoulli random variables with parameter $p$. Thus, for each edge $e\in E$, the edge $e$ is open when $\pi_e=1$ and closed when $\pi_e=0$. A deterministic percolation configuration will be denoted by $\omega\in\{0,1\}^E$. The probability space may contain additional randomness besides the percolation configuration $\pi$, as the events appearing below are not assumed to be measurable with respect to $\pi$. 

Whenever $X$ is a random variable measurable with respect to
$\sigma(\pi)$, we regard $X$ as a function of the percolation
configuration and write $X(\omega)$ for its value when $\pi=\omega$.

The degenerate cases $p\in\{0,1\}$ are excluded only to avoid immaterial conventions concerning conditional probabilities. 

The exclusion of loops is only a matter of convenience. Since loops do not affect connectivity, their states may be incorporated into the additional randomness of the underlying probability space.

For vertices $u,v\in V$, we write $\{u\conn v\}$ for the event that
$u$ and $v$ are connected by an open path in the random configuration
$\pi$. We define the event
\begin{equation*}
\{\ori\conn\bo\}:=\bigcup_{b\in\bo}\{\ori\conn b\}.
\end{equation*}

\begin{definition}
Let $A$ be an event of the underlying probability space. We say that $A$ is \emph{increasing with respect to the percolation configuration} if the function $\omega \mapsto \pr[A\mid \pi=\omega]$ is increasing on $\{0,1\}^E$.
\end{definition}
This notion extends the usual notion of an increasing event. Indeed, if $A$ is measurable with respect to $\sigma(\pi)$, then $\pr[A\mid \pi=\omega]=\mathbf 1_A(\omega)$ for every configuration $\omega$, and the two definitions coincide.

We may now state the main conjecture.

\begin{conj}[Main conjecture]
\label{main_conj}
There exists a universal function $\vph:(0,+\infty)\to\R,$ independent of $G$, $\ori$, $\bo$ and $p$, such that $\vph(\eps)\to1$ as $\eps\to0$, and such that the following statement always holds.

Let $\eps>0$. Let $(E_b)_{b\in\bo}$ be a family of events which are increasing with respect to the percolation configuration. Assume that
$\pr[E_b]\ge 1-\eps$ for every $b \in \bo$. Then
\begin{equation*}
\pr\Biggl[
\bigcup_{b\in\bo}
\Bigl(\{\ori\conn b\}\cap E_b\Bigr)
\Biggr]
\ge
\vph(\eps)\,
\pr[\ori\conn\bo].
\end{equation*}
\end{conj}
Whenever $\pr[\ori\conn\bo]>0$, the conclusion may equivalently be written as
\begin{equation}
\label{mainthmeq2}
\pr\left[
\bigcup_{b\in\bo}
\Bigl(\{\ori\conn b\}\cap E_b\Bigr)
\middle|
\ori\conn\bo
\right]
\ge
\vph(\eps).
\end{equation}

The conjecture may be interpreted as follows. Suppose that each boundary vertex $b$ is equipped with an event $E_b$ which is both highly probable and increasing with respect to the percolation configuration. Then, conditioned on the existence of a connection from the origin to the boundary, the open cluster of the origin intersects, with high probability, the collection of boundary vertices whose associated event occurs.

We shall use the following terminology throughout the paper. We say that a boundary vertex $b\in\bo$ is a \emph{connection} if $\ori\conn b$. A connection $b\in\bo$ will be called a \emph{success} if $E_b$ occurs. In this terminology, the conjecture states that whenever the events $(E_b)_{b\in\bo}$ are sufficiently likely, the existence of a connection almost guarantees the existence of a success.

We now state the main results.
\begin{theorem}[Planar case]
\label{planar_thm}
Assume that $G$ is equipped with a fixed planar embedding and that
$\{\ori\}\cup \bo$ is contained in the boundary of the outer face. Then Conjecture~\ref{main_conj} holds with $\vph(\eps)=1-2\sqrt{\eps}.$
\end{theorem}
No attempt is made here to optimize $\vph(\eps)$. What is important for our purposes is that it depends only on $\eps$ and converges to $1$ as $\eps\to0$.

For Bernoulli bond percolation on $\Z^d$, let
\[
\theta(p)=\pr_p[0\conn\infty]
\text{ for all } p \in [0,1] \qquad\text{and}\qquad 
p_c=\inf\{p\in[0,1]:\theta(p)>0\}.
\]

\begin{theorem}[Conjecture $\implies$ $\theta(p_c)=0$ on $\Z^d$]
\label{thm_theta_pc}
If Conjecture~\ref{main_conj} holds, then for every $d\geq2$,
$\theta(p_c)=0$. Equivalently, for every $d\geq2$, critical Bernoulli
bond percolation on $\Z^d$ has almost surely no infinite open cluster.
\end{theorem}

Theorem~\ref{thm_theta_pc} already follows from the work of
Kozma and Nitzan~\cite{kozma2024reductionthetapc0}, together with
Conjecture~\ref{main_conj}; we give the short deduction below.
We nevertheless intend to include our own proof in a subsequent
version. Although the two arguments are close at the conceptual
level, our proof differs in several details, and including it will
also make the paper self-contained.

\begin{proof}[Proof of Theorem~\ref{thm_theta_pc} via Theorem~6 of Kozma and Nitzan]
Kozma and Nitzan introduced a high-probability connectivity conjecture,
stated as Conjecture~3 in their paper. In the finite constant-parameter
setting, it asserts that for every $\eta>0$ there exists
$\delta=\delta(\eta)>0$ such that, for every finite graph $G=(V,E)$
endowed with Bernoulli bond percolation with a common edge parameter
$p$, every $A\subset V$, and every $o,z\in V$, if
\[
    \pr[o\conn A]>1-\delta
    \qquad\text{and}\qquad
    \pr[a\conn z]>1-\delta
    \quad\text{for every }a\in A,
\]
then $\pr[o\conn z]>1-\eta$.

We show that Conjecture~\ref{main_conj} implies this statement.
Fix $\eta>0$. Since $\vph(\delta)\to1$ as $\delta\to0$, we may choose
$\delta>0$, depending only on $\eta$, sufficiently small that
$\vph(\delta)>0$ and
$\vph(\delta)(1-\delta)>1-\eta$.

Take $\bo=A$ and, for each $a\in A$, let $E_a=\{a\conn z\}$.
These events are increasing. If
$\pr[a\conn z]>1-\delta$ for every $a\in A$, then in particular
$\pr[E_a]\geq1-\delta$, and Conjecture~\ref{main_conj} gives
\[
\pr\left[
    \bigcup_{a\in A}
    \bigl(\{o\conn a\}\cap\{a\conn z\}\bigr)
\right]
\geq
\vph(\delta)\pr[o\conn A].
\]
Moreover,
\[
\bigcup_{a\in A}
    \bigl(\{o\conn a\}\cap\{a\conn z\}\bigr)
=
\{o\conn A\}\cap\{o\conn z\}.
\]
Hence, if $\pr[o\conn A]>1-\delta$, then
\[
    \pr[o\conn z]
    \geq \vph(\delta)\pr[o\conn A]
    > \vph(\delta)(1-\delta)
    > 1-\eta.
\]
Thus the finite constant-parameter version of Conjecture~3 of
Kozma and Nitzan holds.\footnote{The restriction to finite graphs
with a common edge parameter is inessential. The infinite-graph case
follows by finite-volume approximation, while arbitrary edge parameters
can be treated by approximating each edge by a finite constant-parameter
network and passing to the limit.}
Hence Conjecture~3 holds in the generality in which it is stated in
\cite{kozma2024reductionthetapc0}. Their Theorem~6 therefore yields
$\theta(p_c)=0$ for every $d\geq2$.
\end{proof}

\subsection{The conjecture is not a consequence of monotonicity alone}

The events $E_b$ appearing in Conjecture~\ref{main_conj} are increasing with respect to the percolation configuration, and the connectivity events $\{\ori \conn b\}$ are increasing as well. A natural question is therefore whether the conjecture could follow solely from monotonicity considerations. The purpose of this subsection is to show that this is not the case.

More precisely, the analogue of Conjecture~\ref{main_conj} obtained by replacing the connectivity events $\{\ori\conn b\}$ with arbitrary increasing events is false. Thus, if Conjecture~\ref{main_conj} is true, its validity must rely on structural properties of connectivity events which go beyond monotonicity.

In fact, the analogue of (\ref{mainthmeq2}) for arbitrary increasing events can fail in a dramatic way: even when all the events $(E_b)_{b\in\bo}$ have probability arbitrarily close to one, the corresponding conditional probability may be arbitrarily small.

\begin{proposition}
For every $\varepsilon,\varepsilon'>0$, there exist a non-empty finite graph $G=(V,E)$, $\ori \in V$, $\bo \subseteq V$, a parameter $p\in[0,1]$, and two families of increasing events $(C_b)_{b \in \bo}$ and $(E_b)_{b \in \bo}$ such that $\pr\left[\cup_{b \in \bo} C_b\right]>0$ and
\[
\pr\left[
\bigcup_{b\in\bo}
\Bigl(C_b\cap E_b\Bigr)
\middle|
\bigcup_{b\in\bo} C_b
\right]
\le\eps',
\]
while
\[
\forall b\in\bo,\qquad
\pr[E_b]\ge1-\eps.
\]
\end{proposition}

\begin{proof}
Let $n\in\mathbb N$. Set $p=1/2$. We consider a graph $G=(V,E)$ having $5n$ edges and at least $\binom{5n}{2n}$ vertices. We choose a subset $B\subseteq V$ with
$|\bo|=\binom{5n}{2n},$ and fix a bijection $b\longmapsto S_b$
from $\bo$ onto the collection of all subsets of $E$ containing exactly
$2n$ edges. For every $b\in \bo$, define
\[
    C_b=\{\text{all edges in }S_b\text{ are open}\},
    \qquad
    E_b=\{\text{at least }n\text{ edges in }E\setminus S_b
    \text{ are open}\}.
\]
Both events are increasing. Moreover,
\begin{equation*}
\pr [E_b]=\pr[X_n \ge n], \qquad X_n \sim \mathrm{Bin}(3n,\tfrac12).
\end{equation*}
Since the mean of $X_n$ is $3n/2$, standard concentration estimates imply that
\[
\pr[X_n\ge n]\longrightarrow1
\qquad\text{as }n\to\infty.
\]
Thus,
\[
\pr[E_b]\longrightarrow1
\qquad\text{as }n\to\infty,
\]
uniformly in $b\in\bo$. On the other hand,
\begin{equation*}
\bigcup_{b\in \bo} C_b=\{\text{at least }2n\text{ edges are open}\},
\end{equation*}
since the existence of $2n$ open edges is equivalent to the existence of a boundary vertex $b$ whose associated subset is entirely open. Similarly,
\begin{equation*}
\bigcup_{b\in\bo}(C_b\cap E_b)=\{\text{at least }3n\text{ edges are open}\}.
\end{equation*}

Therefore, 
\begin{equation*}
\pr\left[
\bigcup_{b\in\bo}
\Bigl(C_b\cap E_b\Bigr)
\middle|
\bigcup_{b\in\bo} C_b
\right]=\pr[X_n' \ge 3n \mid X_n' \ge 2n],
\end{equation*}
where $X_n'\sim\mathrm{Bin}\left(5n,\tfrac12\right)$. Since the mean of $X_n'$ is $5n/2$, standard concentration estimates imply that the conditional probability above converges to $0$ as $n\to\infty$.

Hence, given $\varepsilon,\varepsilon'>0$, we may choose $n$
sufficiently large so that, simultaneously,
\[
\pr[E_b]\ge 1-\varepsilon
\qquad\text{for every }b\in\bo,
\]
and
\[
\pr\left[
\bigcup_{b\in\bo}(C_b\cap E_b)
\,\middle|\,
\bigcup_{b\in\bo}C_b
\right]
\le \varepsilon'.
\]
This concludes the proof.
\end{proof}

The proposition shows that the connectivity structure in Conjecture~\ref{main_conj} is essential.

%% file: Preuve_planaire.tex
\section{Proof of Theorem \ref{planar_thm}}

For the remainder of this section, fix a parameter $p \in (0,1)$, $\varepsilon>0$ and a family $(E_b)_{b\in \bo}$ of events that are increasing with respect to the percolation configuration and satisfy $\mathbb{P}[E_b]\geq 1-\varepsilon$ for every $b \in \bo$.

\subsection{Geometric setup and left-first DFS}

Throughout this section, we assume that $G$ is planar and fix a planar
embedding of $G$. Replacing $G$ by the connected component containing
$\ori$, and $\bo$ by its intersection with this component, we may assume
without loss of generality that $G$ is connected. By hypothesis, the
origin $\ori$ and every vertex of $\bo$ lie on the boundary of the outer
face. Since the conclusion is immediate when $\ori\in \bo$, we assume for
convenience throughout this section that $\ori\notin \bo$.

We emphasize that $\bo$ is not assumed to contain all the vertices on
the boundary of the outer face. Apart from the assumptions above, it is
an arbitrary subset of these vertices.

Choose once and for all an occurrence of $\ori$ in the clockwise boundary
walk of the outer face. Starting from this occurrence, traverse the
outer-face boundary clockwise. If a vertex of $\bo$ is encountered
several times during the traversal, only its first occurrence is
retained. This defines a linear order on the vertices of $\bo$, in which
every vertex appears exactly once
(see Figure~\ref{fig:vertex_order}). We refer to this order as the
\emph{boundary order}.

For $b,b'\in B$, we say that $b'$ lies strictly to the right of $b$ if $b'$
occurs after $b$ in the boundary order.

\begin{figure}[ht]
\centering
\includegraphics[ trim=0.5cm 1cm 1.5cm 0.5cm,
  clip, width=1\textwidth]{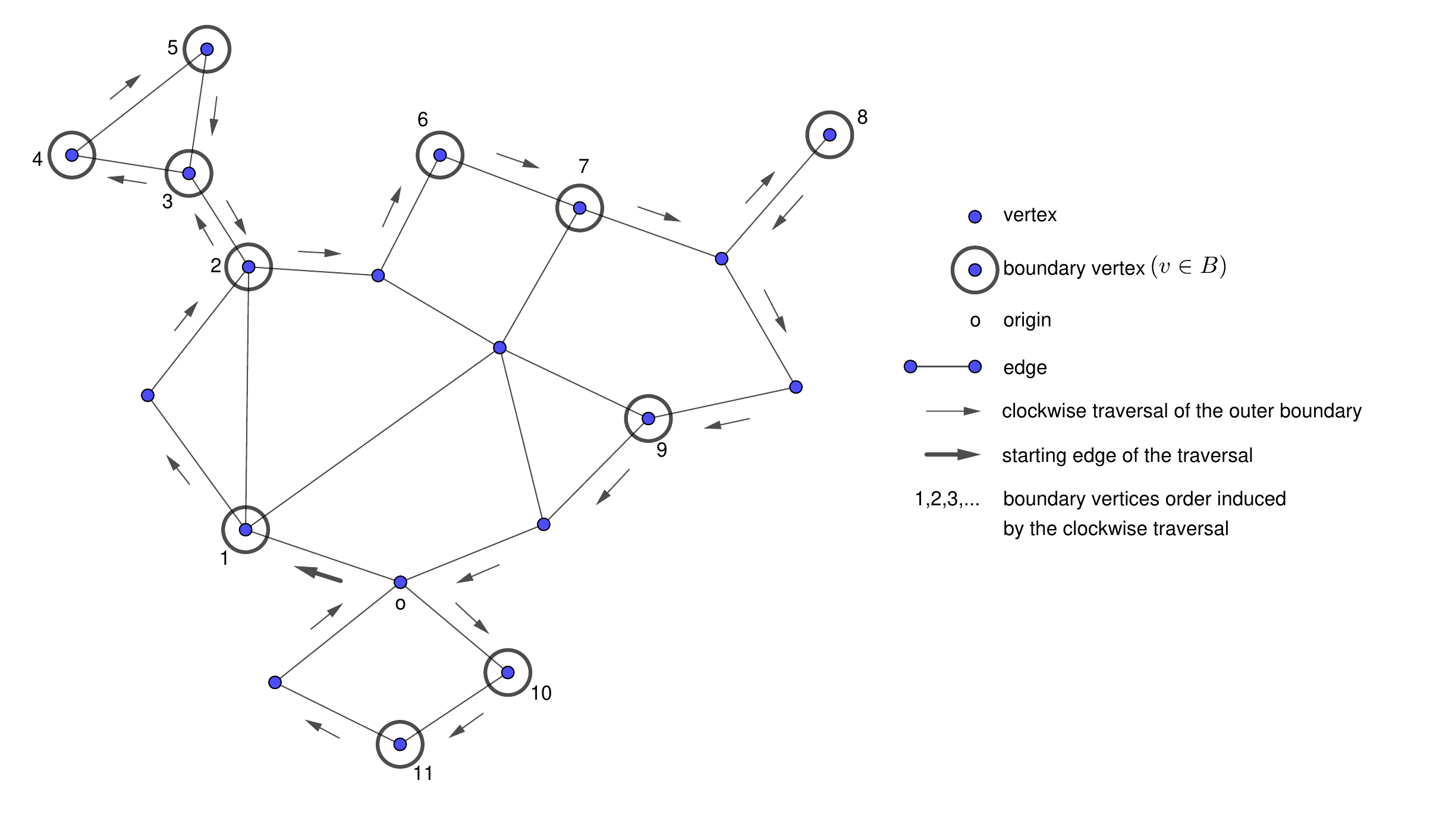}
\caption{
Clockwise order of the boundary vertices. The outer-face boundary is
traversed clockwise starting from the fixed occurrence of the origin
$\ori$, with the initial step indicated by the thick arrow. Boundary
vertices are numbered according to their first encounter along this
traversal.
}
\label{fig:vertex_order}
\end{figure}

We now define the depth-first search used throughout the proof. The
exploration starts at the origin $\ori$. The edges incident to $\ori$ are
taken in clockwise order, beginning with the edge traversed immediately
after the chosen occurrence of $\ori$ in the outer-face boundary walk
(see Figure~\ref{fig:order_origin}).

Suppose that the exploration reaches a vertex $v$ through an edge $e$.
The remaining edges incident to $v$ are taken in clockwise order,
beginning with the edge immediately to the left of $e$
(see Figure~\ref{fig:other_vertex_order}).

If the next edge in this order leads to a previously discovered vertex,
it is skipped. Otherwise, the edge is explored, meaning that its state
is revealed. If the edge is closed, the exploration does not cross it.
If the edge is open, its other endpoint is declared discovered and the
exploration continues recursively from that vertex.

When all the edges incident to the current vertex have been treated,
the exploration backtracks according to the usual depth-first search
rule. We refer to the resulting exploration procedure as the
\emph{left-first depth-first search}, or \emph{left-first DFS}.

Time is indexed by the number of explored edges. Thus, exploring an
edge increases time by one, whereas skipping an edge leading to a
previously discovered vertex does not increase time.

An example of a left-first DFS exploration is shown in
Figure~\ref{fig:DFS_example}.

\begin{figure}[H]
\centering
\includegraphics[ trim=2cm 3.5cm 9cm 4cm, clip, width=0.85\textwidth]{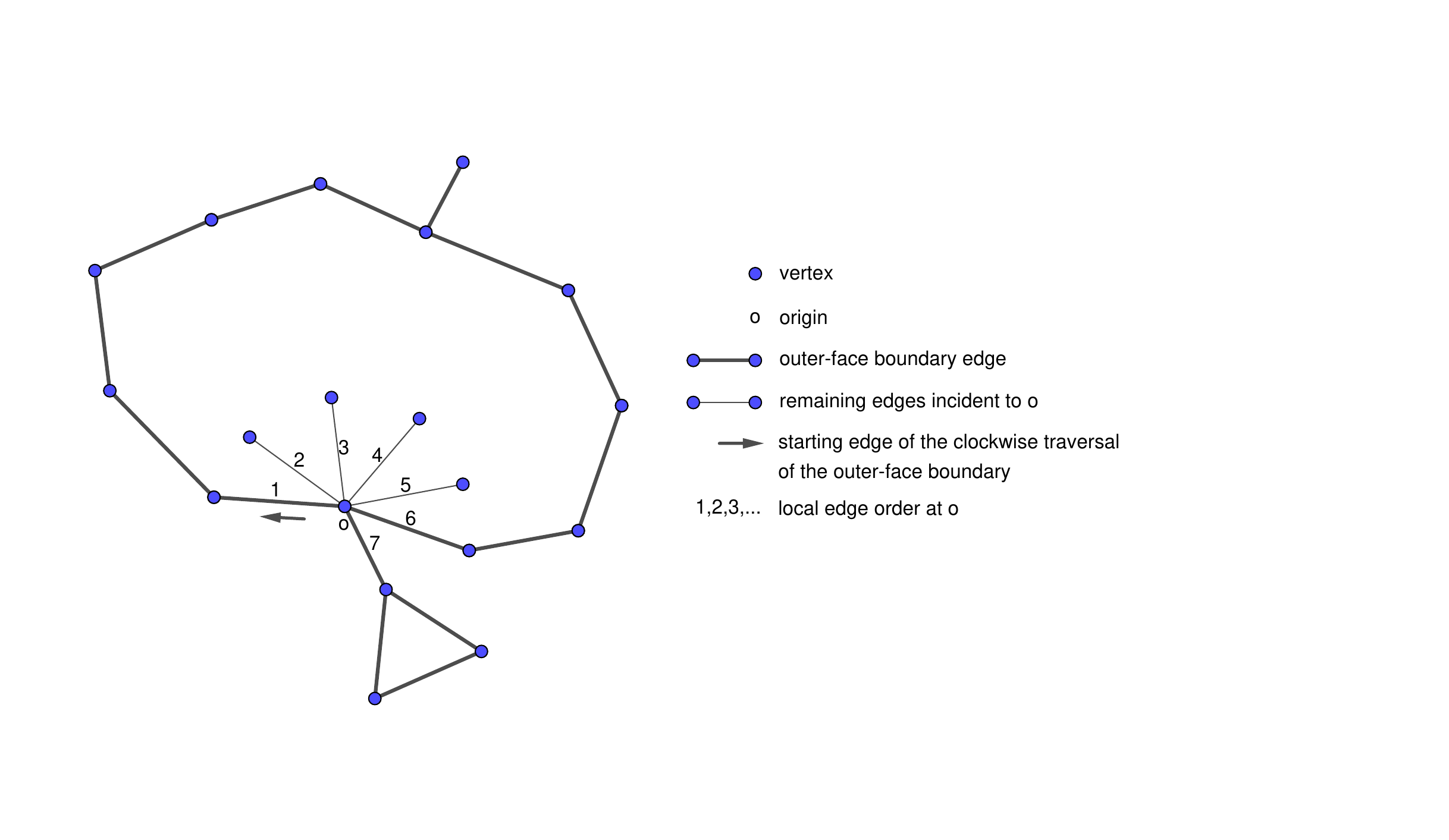}
\caption{
Local ordering of the edges incident to the origin. The labels indicate the order in which these edges are considered by the left-first DFS.
}
\label{fig:order_origin}
\end{figure}

\begin{figure}[H]
\centering
\includegraphics[ trim=2.7cm 6.2cm 7.0cm 1.5cm, clip, width=0.85\textwidth]{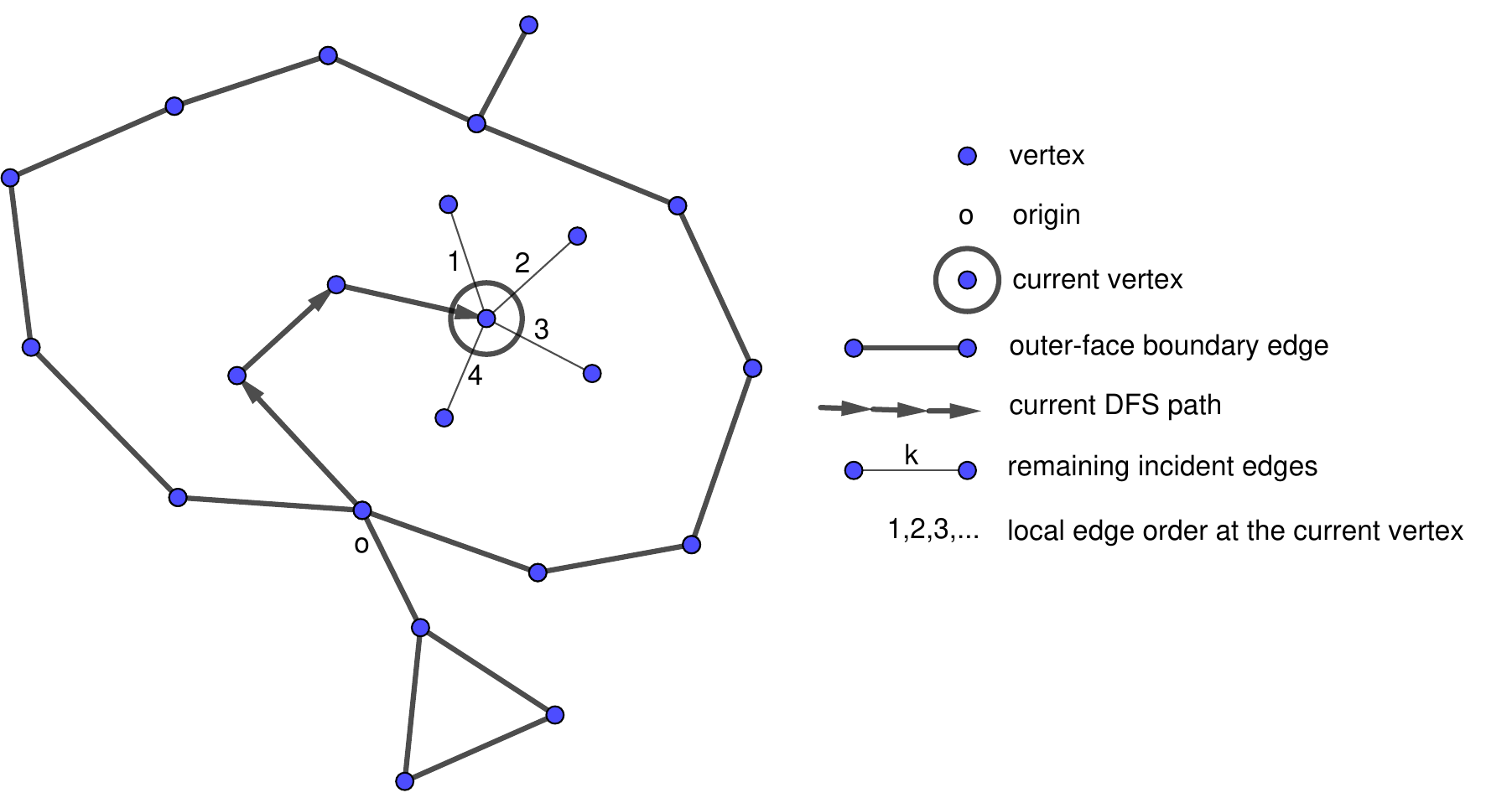}
\caption{
Local ordering of the remaining edges incident to the current vertex. The labels indicate the order in which these edges are considered by the left-first DFS.
}
\label{fig:other_vertex_order}
\end{figure}

\begin{figure}[ht]
\centering
\includegraphics[ trim=3.3cm 1.3cm 2.5cm 0.3cm, clip, width=1\textwidth]{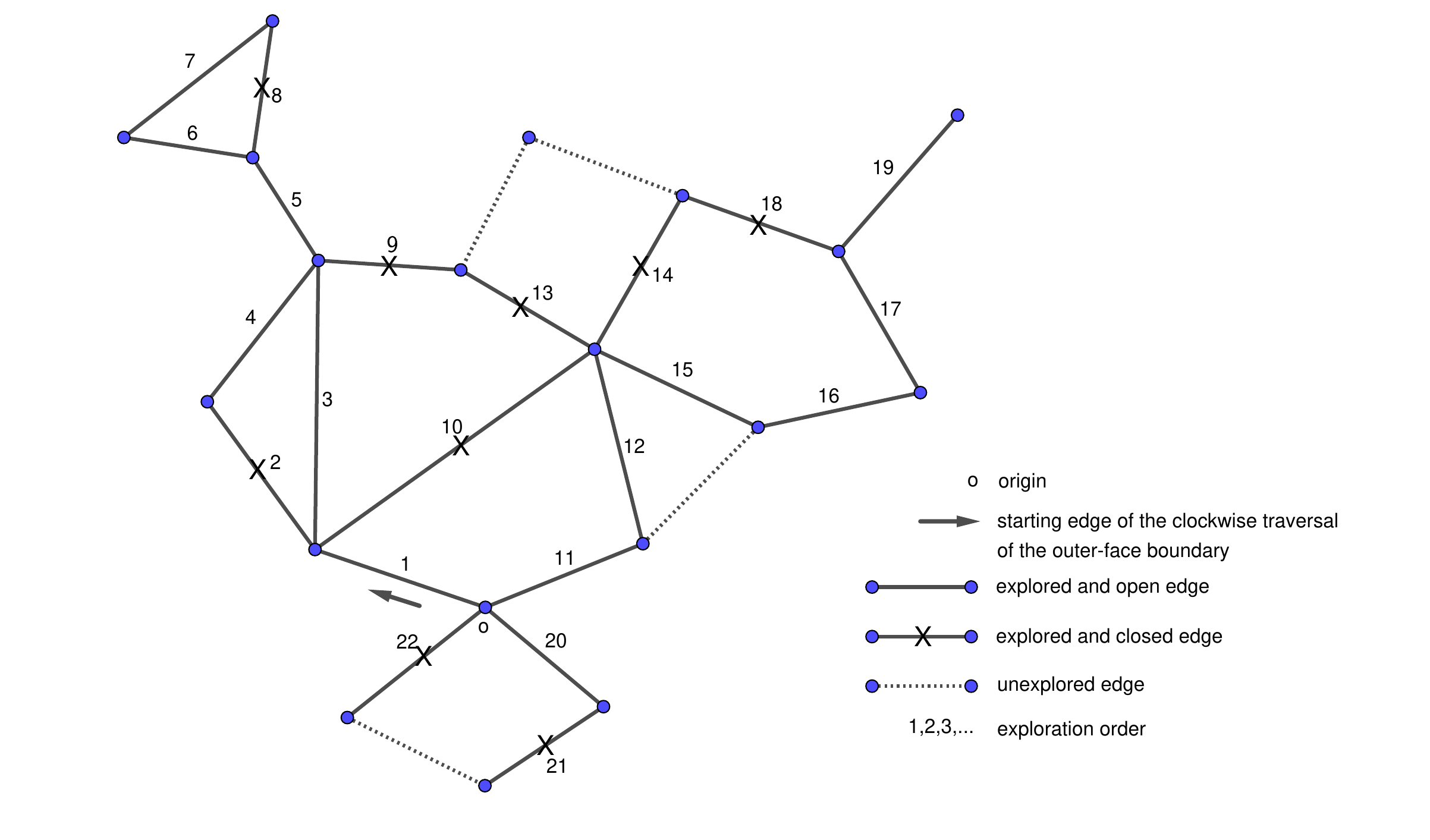}
\caption{
Example of a left-first DFS exploration. Explored edges are shown in bold; closed edges are additionally marked with a cross, while unexplored edges are shown as dotted lines. The labels indicate the order in which edges are explored.
}
\label{fig:DFS_example}
\end{figure}

The open edges through which new vertices are discovered form a tree
rooted at $\ori$, which we call the \emph{DFS exploration tree}.

For every time $t\geq 0$, let $\mathcal F_t$ denote the
$\sigma$-algebra generated by the exploration up to time $t$.
Equivalently, $\mathcal F_t$ contains exactly the information revealed
by the left-first DFS up to time $t$. In particular, $\mathcal F_t\subseteq\sigma(\pi)$ for every $t\ge 0$. Throughout this section, all stopping times are understood to be with respect to the filtration $(\mathcal F_t)_{t\geq 0}$.

For every time $t\geq 0$, let $\mathcal H_t$ denote the
\emph{history set} at time $t$, namely the set of edges explored by
the left-first DFS up to time $t$.

Since at most $|E|$ edges can be explored, the filtration and the
history set are constant after time $|E|$. We use the conventions $\mathcal F_\infty:=\mathcal F_{|E|}$ and $\mathcal H_\infty:=\mathcal H_{|E|}$.

If $\tau$ is a stopping time with respect to the filtration
$(\mathcal F_t)_{t\geq 0}$, we define the corresponding random history
set by $\mathcal H_\tau(\omega):=\mathcal H_{\tau(\omega)}(\omega)$,
using the convention above when $\tau(\omega)=\infty$.

For $b\in \bo$, let $\tau_b$ be the time at which $b$ is discovered,
that is, the time immediately after the exploration of the open edge
by which $b$ is first reached. If $b$ is never reached, we set
$\tau_b=\infty$. Then $\tau_b$ is a stopping time.

\FloatBarrier
\subsection{Main ingredients}

Let $\mathrm{undef}$ be a cemetery value distinct from every element of $\bo$. We define the random variable $\blast\colon\Omega\to\bo\cup\{\mathrm{undef}\}$ as follows: if at least one boundary vertex is connected to $\ori$, then $\blast$ is the last such vertex in the boundary order defined above; otherwise, $\blast=\mathrm{undef}$. We refer to $\blast$ as the \emph{rightmost connection}.

\begin{definition}
Let $b\in\bo$. We call $b$ a \emph{conditional success} if
$b$ is connected to $\ori$ and
\[
\pr[E_b\mid\mathcal{F}_{\tau_b}]
\ge
1-\sqrt{\eps}.
\]
\end{definition}

We use the expression ``$\blast$ is a conditional success'' for the event
\[
\bigcup_{b\in\bo}
\left(
    \{\blast=b\}
    \cap
    \{b\text{ is a conditional success}\}
\right),
\]
and the expression ``$\blast$ is not a conditional success'' for its
complement. The proof of the planar theorem relies on the following two lemmas.

\begin{lemma}
\label{lemmacond}
\begin{equation*}
\pr\bigl[\text{there exists a success}\bigr]
\ge
\left(1-\sqrt{\eps}\right)
\pr\bigl[\text{there exists a conditional success}\bigr].
\end{equation*}
\end{lemma}

\begin{lemma}
\label{lemmarightmostconn}
\begin{equation*}
\pr\bigl[\text{$\blast$ is a conditional success}\bigr]
\ge
\left(1-\sqrt{\eps}\right)\pr[\ori\conn\bo].
\end{equation*}
\end{lemma}

We first show how Theorem \ref{planar_thm} follows from these two lemmas.

\begin{proof}[Proof of Theorem \ref{planar_thm}]
\begin{align*}
\pr[\text{there exists a success}]
&\ge
\left(1-\sqrt{\eps}\right)
\pr\bigl[\text{there exists a conditional success}\bigr]
&&\text{by Lemma~\ref{lemmacond}}
\\
&\ge
\left(1-\sqrt{\eps}\right)
\pr\bigl[\text{$\blast$ is a conditional success}\bigr]
\\
&\ge
\left(1-\sqrt{\eps}\right)^2
\pr[\ori\conn\bo]
&&\text{by Lemma~\ref{lemmarightmostconn}}
\\
&\ge
\left(1-2\sqrt{\eps}\right)
\pr[\ori\conn\bo].
\end{align*}
This proves Theorem \ref{planar_thm}.
\end{proof}

Lemma~\ref{lemmarightmostconn} ensures that a conditional
success exists with high probability, using the rightmost connection as
a witness. The rightmost connection itself should not be expected to be
a success with high probability, since the event $\{\blast=b\}$ carries
additional information about the unexplored part of the configuration.
Lemma~\ref{lemmacond} converts the existence of a
conditional success into the existence of an actual success by
considering the first conditional success reached by the exploration.

We now prove Lemma~\ref{lemmacond}. 
\begin{proof}[Proof of Lemma~\ref{lemmacond}]
Let $\bfirst$ be the $\bo\cup\{\mathrm{undef}\}$-valued random variable
defined as follows. If a conditional success is encountered during the
exploration, then $\bfirst$ is the first such vertex; otherwise,
$\bfirst=\mathrm{undef}$. 

We use the expression ``$\bfirst\text{ is a success}$''
for the event
\begin{equation*}
\bigcup_{b \in \bo} \{\bfirst=b \}\cap \{b \text{ is a success}\}.
\end{equation*}
Let $\tau_{\bfirst}$ be the hitting time of the vertex
$\bfirst$, with the convention that $\tau_{\bfirst}=\infty$ if
$\bfirst=\mathrm{undef}$.
This is a stopping time, since at each time $t$ one can determine from
$\mathcal F_t$ whether a conditional success has already been reached. Therefore,
\begin{equation*}
\pr[\text{there exists a success}]
\ge
\pr\bigl[\bfirst\text{ is a success}\bigr]
=
\ex\Bigl[
\pr\Bigl[
\bfirst\text{ is a success}
\,\Bigm|\,
\mathcal{F}_{\tau_{\bfirst}}
\Bigr]
\Bigr].
\end{equation*}

On the event $\{\bfirst\neq\mathrm{undef}\}$,
\[
\pr\left[
\bfirst\text{ is a success}
\,\big|\,
\mathcal{F}_{\tau_{\bfirst}}
\right]
\ge
1-\sqrt{\eps},
\]
hence
\begin{align*}
\pr\left[\text{there exists a success}\right]
&\ge
\left(1-\sqrt{\eps}\right)
\pr\left[\bfirst\neq\mathrm{undef}\right] \\
&=\left(1-\sqrt{\eps}\right)\pr[\text{there exists a conditional success}]
\end{align*}
which concludes the proof.
\end{proof}

\subsection{Increasing events up to a stopping time}

The proof of Lemma~\ref{lemmarightmostconn} relies on a notion of
partial monotonicity adapted to a stopping time. We begin by
introducing this notion.

\begin{notation}
For a configuration $\omega$, an edge $e$, and $i\in\{0,1\}$,
we denote by $\omega^{(e\to i)}$ the configuration obtained from
$\omega$ by forcing the state of $e$ to be equal to $i$:
\[
\omega^{(e\to i)}(e')
=
\begin{cases}
i, & \text{if } e'=e,\\
\omega(e'), & \text{if } e'\neq e,
\end{cases}
\qquad \text{ for every } e'\in E.
\]
\end{notation}

\begin{definition}
\label{defcroiss}
Let $\tau$ be a stopping time, and let $X$ be a real-valued random
variable measurable with respect to $\sigma(\pi)$.
We say that $X$ is \emph{increasing up to $\tau$} if, for every
configuration $\omega\in\{0,1\}^E$ and every edge
$e\in \mathcal{H}_\tau(\omega)$,
\[
X(\omega^{(e\to0)})
\le
X(\omega^{(e\to1)}).
\]

If $A$ is an event measurable with respect to $\sigma(\pi)$, we say
that $A$ is \emph{increasing up to $\tau$} if its indicator function
$\ind_A$ is increasing up to $\tau$.
\end{definition}

The following elementary observation will be used below.

\begin{remark}
\label{remcroiss1}
For every configuration $\omega$ and every edge $e$, whether $e$ is explored by time $\tau$ does not depend on the state of $e$. More
precisely,
\[
e\in \mathcal{H}_\tau(\omega^{(e\to0)})
\quad\Longleftrightarrow\quad
e\in \mathcal{H}_\tau(\omega^{(e\to1)}).
\]
\end{remark}

\begin{remark}
\label{remcroiss2}
This notion should not be confused with the monotonicity of the
conditional expectation. A random variable may be increasing up to $\tau$ in the sense of
Definition~\ref{defcroiss} while
$\ex[X\mid\mathcal{F}_\tau]$
is not increasing as a function of the configuration.
\end{remark}

The following proposition is a partial version of the classical FKG
inequality adapted to this notion of monotonicity.

\begin{proposition}[Partial FKG inequality]
\label{PropFKG}
Let $\tau$ be a stopping time. Let $X$ and $Y$ be real-valued random variables measurable with respect to $\sigma(\pi)$ and increasing up to $\tau$.
Then
\begin{equation}
\label{EFKG1}
\ex\left[
X\,\ex\left[Y\mid\mathcal{F}_\tau\right]
\right]
\ge
\ex[X]\ex[Y].
\end{equation}
\end{proposition}

\begin{remark}
\label{rem:FKG1}
Proposition~\ref{PropFKG} admits the equivalent formulation obtained by replacing the left-hand side of~\eqref{EFKG1} with
\[
\ex\Bigl[
\ex[X\mid\mathcal F_\tau]\,
\ex[Y\mid\mathcal F_\tau]
\Bigr].
\]
Indeed, since $\ex[Y\mid\mathcal F_\tau]$ is
$\mathcal F_\tau$-measurable,
\[
\ex\Bigl[
X\,\ex[Y\mid\mathcal F_\tau]
\Bigr]
=
\ex\Bigl[
\ex\Bigl[
X\,\ex[Y\mid\mathcal F_\tau]
\,\Big|\,
\mathcal F_\tau
\Bigr]
\Bigr]
=
\ex\Bigl[
\ex[X\mid\mathcal F_\tau]\,
\ex[Y\mid\mathcal F_\tau]
\Bigr].
\]
We use the formulation~\eqref{EFKG1} because it is precisely the one
needed in the proof of Lemma~\ref{lemmarightmostconn}.
\end{remark}

\begin{remark}
Despite its equivalent formulation, Proposition~\ref{PropFKG}
cannot be obtained by applying the standard FKG inequality directly to
the conditional expectations. Indeed, the standard FKG inequality
cannot be applied directly to
$\ex[X\mid\mathcal{F}_\tau]$
and
$\ex[Y\mid\mathcal{F}_\tau]$,
since these random variables are not increasing in general (see
Remark~\ref{remcroiss2}).
\end{remark}

The proof proceeds by induction and closely follows the classical
proof of the FKG inequality. It is deferred to
Appendix~\ref{app:proofFKG}.

\subsection{Main geometric lemma}

The proof of Lemma~\ref{lemmarightmostconn} relies on the following
geometric property. The proof below is the only point where planarity
is used.

\begin{lemma}
\label{lemcrois}
For every $b\in\bo$, the event $\{\blast=b\}$ is increasing up to
$\tau_b$.
\end{lemma}

\paragraph{Left and right of an oriented path.}
For each $b\in\bo$, let $\beta_b$ denote the portion of the clockwise
outer-face boundary walk from the fixed occurrence of $\ori$ to the
occurrence of $b$ retained in the definition of the boundary order.
Let $\gamma$ be a simple path from $\ori$ to $b$, oriented from $\ori$
to $b$. Concatenating $\gamma$ with $\beta_b$ traversed in the reverse
direction gives an oriented closed walk
$C_\gamma=\gamma\cdot\beta_b^{-1}$, which need not be simple.

A face is called \emph{interior} if the winding number of
$C_\gamma$ around any point of that face is non-zero. An edge lies to
the left of $\gamma$ if it is traversed by $C_\gamma$ or is incident
to an interior face. It lies strictly to the left of $\gamma$ if it
lies to the left of $\gamma$ but does not belong to $\gamma$. An edge
lies to the right of $\gamma$ if it does not lie strictly to its left.
Thus, every edge lies to the left or to the right of $\gamma$, and the
two sides overlap precisely on the edges of $\gamma$. A path is said
to lie to the left or to the right of $\gamma$ if each of its edges
does. These conventions are illustrated in
Figure~\ref{fig:left_right}.

\begin{figure}[H]
\centering
\includegraphics[ trim=4.4cm 5cm 4.8cm 3cm, clip, width=1\textwidth]{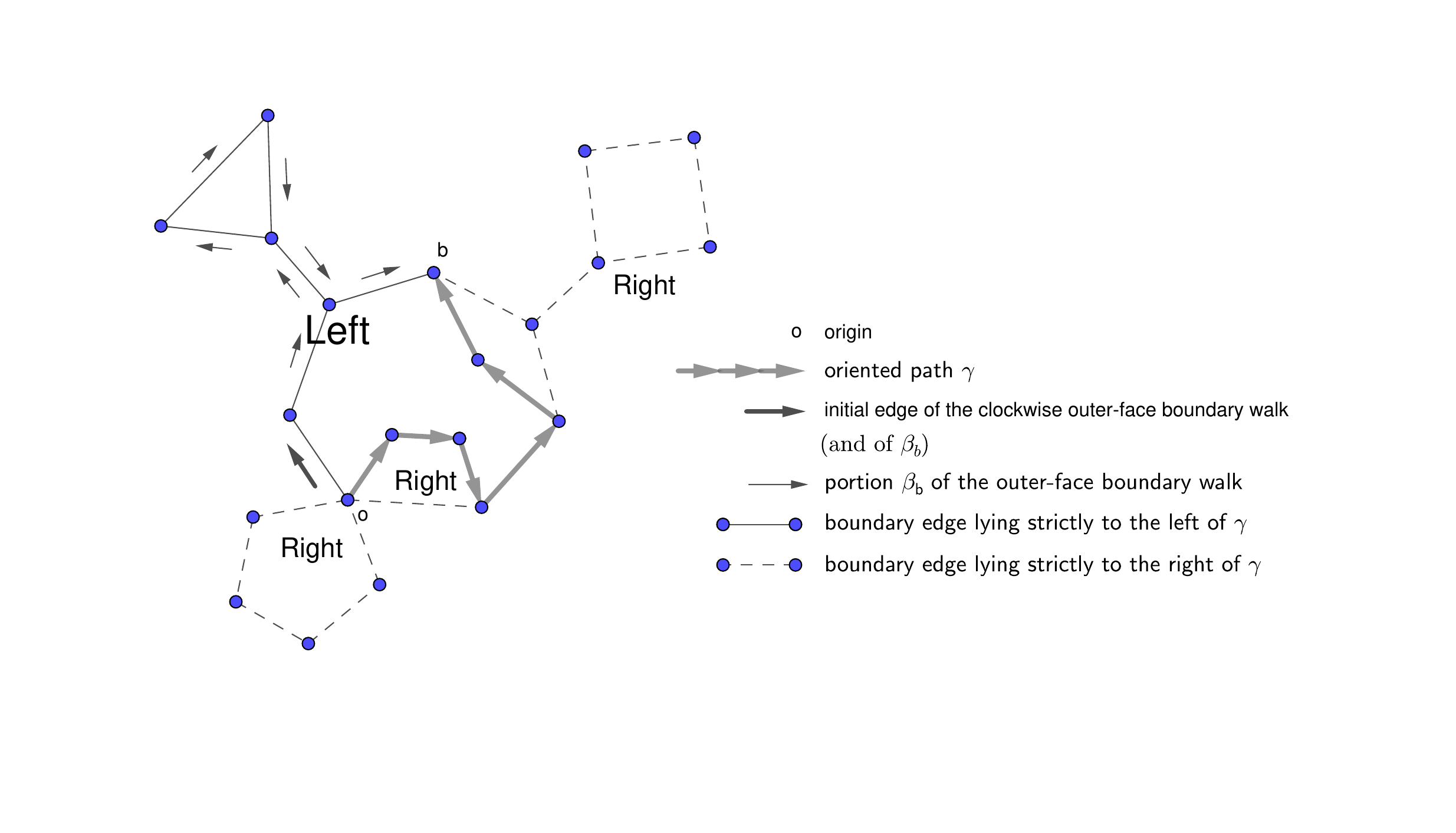}
\caption{
Illustration of the left- and right-hand sides of an oriented path
$\gamma$ from the origin $\ori$ to a boundary vertex $b$. Only the
edges on the boundary of the outer face are shown, together with
$\gamma$; all remaining edges of the graph are omitted. The arrows
along the boundary represent the clockwise boundary walk $\beta_b$
from the fixed occurrence of $\ori$ to the retained occurrence of
$b$. Recall that the closed walk
$C_\gamma=\gamma\cdot\beta_b^{-1}$ traverses $\beta_b$ in the opposite
direction.
}
\label{fig:left_right}
\end{figure}

\medskip
Figure~\ref{fig:historyset} illustrates the key geometric observation underlying the
proof: for every $b\in\bo$, on the event $\{\ori\conn b\}$, every edge
explored by time $\tau_b$ lies to the left of the DFS-tree path from
$\ori$ to $b$. We omit the proof of this property, which follows from
the local ordering defining the left-first DFS together with planarity.

We shall also use the following standard consequence of planarity. Let
$\gamma$ be an open simple path from $\ori$ to $b$. If a boundary
vertex $b'$ lying to the right of $b$ is connected to $\ori$, then
there exists an open path $\gamma'$ from $\ori$ to $b'$ lying to the
right of $\gamma$.

\begin{figure}[H]
\centering
\includegraphics[trim=2cm 1.5cm 1cm 1.9cm, clip,width=1\textwidth]{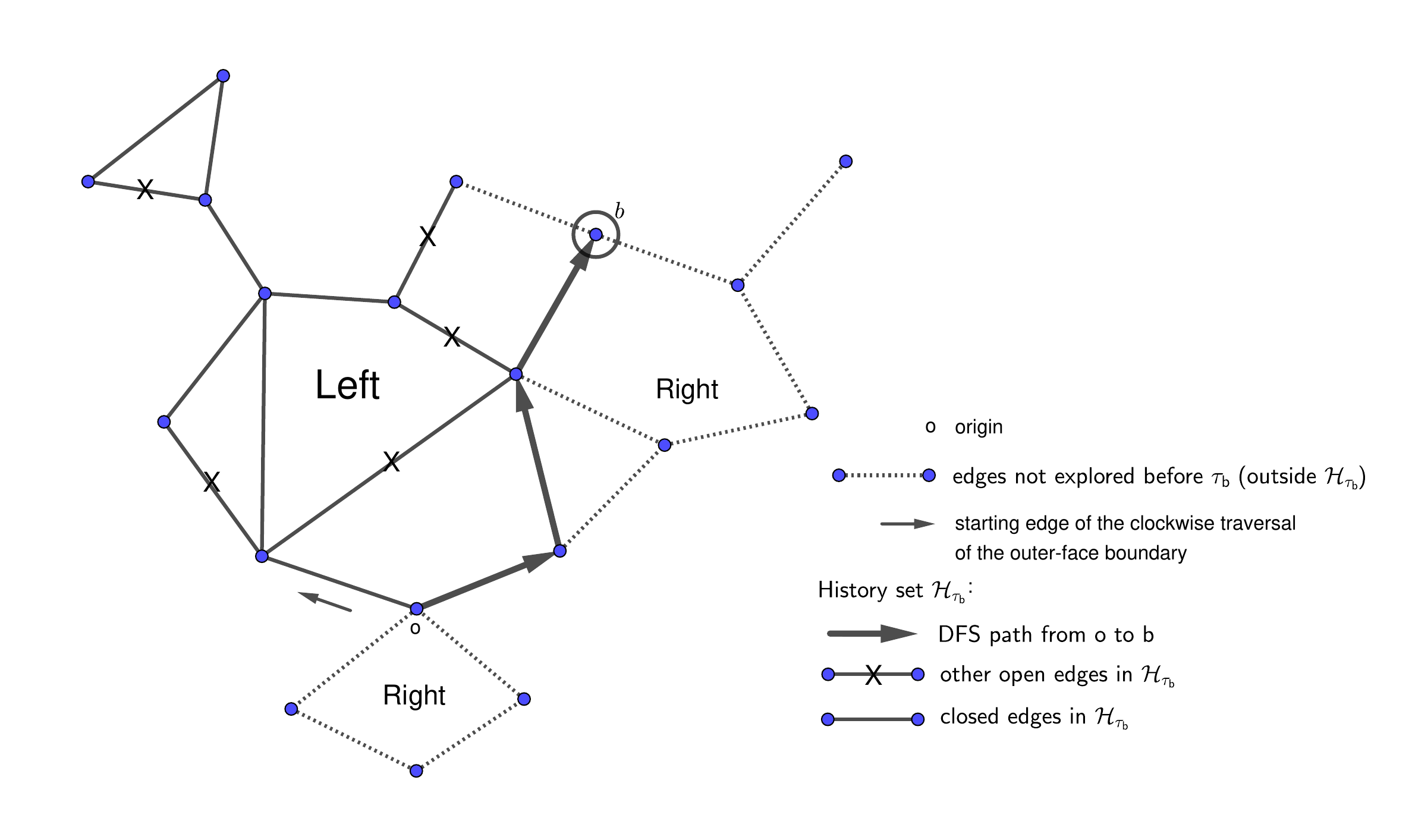}
\caption{Illustration of the key geometric property used in the proof of Lemma~\ref{lemcrois}: on the event $\{\ori\conn b\}$, every edge in the history set $\mathcal{H}_{\tau_b}$ lies to the left of the DFS path from $\ori$ to $b$.}
\label{fig:historyset}
\end{figure}

\begin{proof}[Proof of Lemma \ref{lemcrois}]
Fix $b\in\bo$, let $\omega\in\{0,1\}^E$, and let
$e\in\mathcal H_{\tau_b}(\omega)$. By
Remark~\ref{remcroiss1}, $e\in
\mathcal H_{\tau_b}\bigl(\omega^{(e\to0)}\bigr)$.
It therefore suffices to prove that
\[
\ind_{\{\blast=b\}}\bigl(\omega^{(e\to0)}\bigr)
\le
\ind_{\{\blast=b\}}\bigl(\omega^{(e\to1)}\bigr).
\]

Assume for contradiction that $\omega^{(e\to0)}$ satisfies
$\{\blast=b\}$, whereas $\omega^{(e\to1)}$ does not. Since $b$ is
connected to $\ori$ in $\omega^{(e\to0)}$, let $\gamma$ denote the
unique path from $\ori$ to $b$ in the corresponding DFS exploration
tree, oriented from $\ori$ to $b$.

By the geometric observation above, the edge $e$ lies to the left of
$\gamma$. Since $e$ is closed in $\omega^{(e\to0)}$, it cannot belong
to $\gamma$. Consequently, $e$ lies strictly to the left of $\gamma$.

The path $\gamma$ remains open in $\omega^{(e\to1)}$, and hence $b$
remains connected to $\ori$ in this configuration. Since
$\omega^{(e\to1)}$ does not satisfy $\{\blast=b\}$, there exists a
boundary vertex $b'$ lying strictly to the right of $b$ which is connected to
$\ori$ in $\omega^{(e\to1)}$.

By the planar observation above, there exists an open path $\gamma'$
from $\ori$ to $b'$ lying to the right of $\gamma$. Since $e$ lies
strictly to the left of $\gamma$, the path $\gamma'$ cannot contain
$e$. Therefore, $\gamma'$ is also open in $\omega^{(e\to0)}$. Thus
$b'$ is connected to $\ori$ in $\omega^{(e\to0)}$, contradicting the
assumption that $\blast=b$ in this configuration.

This proves the desired monotonicity.
\end{proof}

\subsection{Proof of Lemma~\ref{lemmarightmostconn}}

\begin{proof}
Fix $b \in \bo$. Since $E_b$ is increasing with respect to the percolation configuration, the function $\omega\longmapsto\pr[E_b\mid\pi=\omega]$ is increasing on $\{0,1\}^E$, and therefore increasing up to $\tau_b$. Moreover, by Lemma~\ref{lemcrois},
$\ind_{\{\blast=b\}}$ is increasing up to $\tau_b$.
Applying Proposition~\ref{PropFKG} with $X=\ind_{\{\blast=b\}}$ and $Y=\pr[E_b\mid\pi]$, we obtain
\[
\ex\Bigl[
\ind_{\{\blast=b\}}\,
\ex\Bigl[
\pr[E_b\mid\pi]
\,\Big|\,
\mathcal F_{\tau_b}
\Bigr]
\Bigr]
\ge
\ex\Bigl[
\ind_{\{\blast=b\}}
\Bigr]\,
\ex\Bigl[
\pr[E_b\mid\pi]
\Bigr].
\]
Since $\mathcal F_{\tau_b}\subseteq\sigma(\pi)$, this simplifies to
\[
\ex\Bigl[
\ind_{\{\blast=b\}}
\pr[E_b\mid\mathcal F_{\tau_b}]
\Bigr]
\ge
\pr[\blast=b]\,
\pr[E_b].
\]
Let \[
S_b:=\pr[E_b\mid\mathcal F_{\tau_b}].
\]
Since $\pr[E_b]\ge1-\varepsilon$, we obtain
\[
\ex\Bigl[
\ind_{\{\blast=b\}}S_b
\Bigr]
\ge
(1-\varepsilon)\pr[\blast=b],
\]
or equivalently,
\begin{equation}
\label{EapplyFKG2}
\ex\Bigl[
\ind_{\{\blast=b\}}(1-S_b)
\Bigr]
\le
\varepsilon\,\pr[\blast=b].
\end{equation}
Therefore,
\begin{align}
\label{EapplyFKG3}
&\pr\Bigl[
\{\blast=b\}
\cap
\{\text{$b$ is not a conditional success}\}
\Bigr]
\nonumber\\
&=
\pr\Bigl[
\{\blast=b\}
\cap
\{S_b<1-\sqrt\varepsilon\}
\Bigr]
\nonumber\\
&=
\pr\Bigl[
\ind_{\{\blast=b\}}(1-S_b)>\sqrt\varepsilon
\Bigr]
\nonumber\\
&\le
\frac1{\sqrt\varepsilon}
\ex\Bigl[
\ind_{\{\blast=b\}}(1-S_b)
\Bigr]
\qquad\text{(Markov's inequality)}
\nonumber\\
&\le
\sqrt\varepsilon\,\pr[\blast=b],
\qquad\text{by~\eqref{EapplyFKG2}}.
\end{align}
Here Markov's inequality applies since $\ind_{\{\blast=b\}}(1-S_b)\ge0$. Since the events $\{\blast=b\}$, for $b\in\bo$, are pairwise disjoint and
their union is $\{\ori\conn\bo\}$, summing~\eqref{EapplyFKG3} over all $b\in\bo$ gives
\[
\pr\!\left[
\{\ori\conn\bo\} \cap \{\blast
\text{ is not a conditional success}\}
\right]
\le
\sqrt\varepsilon\,\pr[\ori\conn\bo].
\]
Since
$\{\blast \text{ is a conditional success}\}
\subseteq \{\ori\conn\bo\}$,
the preceding inequality may be rewritten as
\[
\pr\!\left[
\text{$\blast$ is a conditional success}
\right]
\ge
(1-\sqrt\varepsilon)\pr[\ori\conn\bo],
\]
which completes the proof.
\end{proof}

%% file: Reflexions_conjecture.tex
\section{Remarks on the general conjecture}
\label{sec:remarks_general}

The purpose of this section is not to provide a proof of the general
conjecture, but rather to isolate some features of the planar argument
which may be relevant beyond the planar setting, as well as some of the
obstacles encountered when trying to extend it.

\paragraph{A heuristic in favor of the general conjecture.}

We first give a heuristic suggesting that the conjecture may remain true
in general graphs.

Recall that two vertices of a graph are said to be
\emph{2-edge-connected} if they can be joined by two edge-disjoint
paths, with the convention that a vertex is connected to itself by the
empty path. Equivalently, they are connected and remain connected after
the deletion of any single edge. This defines an equivalence relation
on the vertices, whose equivalence classes are called the
\emph{2-edge-connected components}.

Contracting each 2-edge-connected component of a connected graph
produces a tree, whose edges correspond precisely to the bridges of the
original graph. Applied to an open cluster, this decomposition is
particularly relevant for connectivity events. Indeed, on the event
that two vertices are connected, the open pivotal edges for their
connection are precisely the bridges of the open cluster whose deletion
disconnects them.

The abstract structure of this bridge tree does not distinguish planar
graphs from general graphs. Indeed, any finite tree can occur as the
bridge tree of a planar graph; for instance, one may simply take the
graph itself to be a tree. Thus, the planar theorem already allows
arbitrary tree-like arrangements of bridges, and hence arbitrary tree-like organizations of pivotal edges for connectivity events. This suggests
that the complexity of the bridge structure alone should not be an
obstruction to the general conjecture.

This heuristic does not suggest that the bridge tree alone captures all
the relevant information. In particular, any exploration-based approach
must also take into account the way in which connections are discovered
over time. This dynamic information is not determined solely by the
bridge tree of the final open cluster: the same final connectivity
structure may arise through different exploration histories. Thus, if
the decomposition into 2-edge-connected components is relevant to the
general conjecture, it should likely be combined with information
describing the dynamics of an exploration.

\paragraph{Attempts to generalize the planar exploration.}

The proof of the planar case relies on a left-first depth-first
exploration and on the geometric order induced by the planar embedding.
This order is crucial for the definition of the distinguished
connection $b^*$ and for establishing the partial growth property of
Lemma~\ref{lemcrois}, which in turn provides the monotonicity needed for
the partial FKG argument.

In a general graph, the definition of $b^*$ cannot be reproduced
directly, since there is no analogous geometric order on the boundary.
A natural possibility would be to replace it with a distinguished
connection defined solely from the exploration, while requiring this partial growth property to remain valid. We investigated this
possibility extensively, considering several exploration-based
definitions of an object playing the role of $b^*$, but none of them
led to the required partial growth property. We do not expect a direct generalization of the planar proof based solely on such a redefinition of $b^*$ to succeed.

Although this exact property seems to be lost in general graphs, a weaker structure remains. Fix an arbitrary
depth-first exploration, and define $b_\mathrm{last}$ to be the last vertex of $\bo$ discovered by the exploration, whenever such a vertex exists. We use the same notation \(\mathcal H_t\) and \(\tau_b\) for the history set and discovery time associated with this DFS.

For a configuration $\omega$, let
\[
\mathsf{Conn}(\omega)
:=
\{b'\in\bo:\ori\conn b' \text{ in }\omega\}.
\]
Fix $b\in\bo$, a configuration $\omega$, and an edge
$e\in\mathcal H_{\tau_b}(\omega)$. Compare the two configurations
$\omega^{(e\to0)}$ and $\omega^{(e\to1)}$. The following dichotomy
holds.

\begin{itemize}
    \item If $e$ belongs to no open cycle in $\omega^{(e\to1)}$, then $e$ is a bridge of the open graph in $\omega^{(e\to1)}$, and the partial
    growth property is recovered:
    \[
    \ind_{\{b_\mathrm{last}=b\}}\bigl(\omega^{(e\to0)}\bigr)
    \leq
    \ind_{\{b_\mathrm{last}=b\}}\bigl(\omega^{(e\to1)}\bigr).
    \]

    \item If $e$ belongs to an open cycle in $\omega^{(e\to1)}$, then
    deleting $e$ does not change the open cluster of the origin. In
    particular,
    \[
    \mathsf{Conn}\bigl(\omega^{(e\to0)}\bigr)
    =
    \mathsf{Conn}\bigl(\omega^{(e\to1)}\bigr).
    \]
    What may change is only the order in which these connections are
    discovered by the DFS, and therefore possibly the value of $b_\mathrm{last}$.
\end{itemize}

This dichotomy is potentially useful because it isolates the obstruction
to partial growth. In the bridge case, connectivity genuinely grows;
in the cyclic case, the set of boundary connections is unchanged and
only their order of discovery may vary. This suggests that a possible
extension of the planar argument should try to absorb these changes of
order into the induction rather than eliminate them.

One possible direction is therefore to replace the induction around a
single distinguished connection by a global recursion involving all
boundary connections simultaneously. Such a recursion would allow their
relative order to evolve during the exploration while preserving enough
monotonicity for the argument to close.

At present, however, we have not found a satisfactory way to absorb
these changes of order into a global recursive argument.

\paragraph{An adversarial formulation.}

There is another way to isolate the probabilistic mechanism behind the
conjecture. Let $(\mathcal F_t)$ denote the filtration generated by an
exploration, and, for each $b\in\bo$, consider
\[
q_b(t):=\pr[E_b\mid\mathcal F_t].
\]
Initially,
\[
q_b(0)=\pr[E_b]\geq1-\eps.
\]
For each $b$, the process $(q_b(t))_t$ is a martingale. Moreover, since
$E_b$ is increasing with respect to the percolation configuration,
revealing an edge to be open can only increase the corresponding
conditional probability relative to revealing the same edge to be
closed.

More explicitly, suppose that the next explored edge $e_{t+1}$ is
unrevealed. Denote by $q_{b,t}^{(1)}$ and $q_{b,t}^{(0)}$ the two
possible values of $q_b(t+1)$ according as $e_{t+1}$ is revealed open
or closed. Thus,
\[
q_b(t+1)=q_{b,t}^{(\pi_{e_{t+1}}))}.
\]
Since $E_b$ is increasing with respect to the percolation
configuration, $q_{b,t}^{(1)}\geq q_{b,t}^{(0)}$, while the martingale property gives
\[
q_b(t)=p\,q_{b,t}^{(1)}+(1-p)\,q_{b,t}^{(0)}.
\]
In particular,
\(
q_{b,t}^{(0)}\leq q_b(t)\leq q_{b,t}^{(1)}.
\)

This leads to the following informal adversarial interpretation.
Instead of starting from actual events $(E_b)_{b\in\bo}$, imagine an
adversary controlling the quantities $q_b(t)$ during the exploration,
subject only to the martingale and monotonicity constraints above. The
adversary attempts to make $q_b(t)$ small precisely for those boundary
vertices which eventually become relevant connections.

Every family of increasing events gives rise to such a system through
$q_b(t)=\pr[E_b\mid\mathcal F_t]$. The converse need not hold, so this
adversarial model should be regarded as a relaxation of the original
problem rather than as an equivalent formulation.

This adversarial statement is not an immediate consequence of the
planar theorem as stated. However, the same proof, reformulated in terms
of the processes $(q_b(t))_t$, shows that for the left-first depth-first
exploration the adversary cannot achieve its objective under the
martingale and monotonicity constraints described above.

It seems somewhat unnatural that this phenomenon should depend
essentially on the very specific left-first depth-first exploration
used in the planar proof. The conjecture itself does not involve any
exploration process, and this suggests heuristically that the
adversary might remain unable to win for a substantially broader class
of explorations. At present, we do not have a precise statement or
argument in this direction. However, identifying a more intrinsic
mechanism explaining why the adversary cannot win, and which remains
valid beyond the particular planar exploration used here, could provide
a natural route towards the general conjecture.

%% file: Appendice.tex
\appendix
\section{Proof of Proposition~\ref{PropFKG}}
\label{app:proofFKG}

For a stopping time $\tau$, we denote by $\pi_\tau$ the partial
percolation configuration revealed by the exploration up to time
$\tau$. Thus $\pi_\tau(e)\in\{0,1,\star\}$, where $\star$ indicates
that the state of $e$ has not yet been revealed. Finally, for every configuration $\omega$, we denote by
$\omega_\tau:=\pi_\tau(\omega)$ the corresponding realization of the
exploration history up to time $\tau$.

\begin{proof}[Proof of Proposition~\ref{PropFKG}]
Let $\tau$ be a stopping time, and let $X$ and $Y$ be two random
variables measurable with respect to $\sigma(\pi)$, and increasing up to $\tau$.

We will prove by induction on $t \in \N$ that, for every $t\in\mathbb N$,
\begin{equation}
\label{EqHDR}
\ex\Bigl[
\ex[X\mid\mathcal F_{t\wedge\tau}]
\,
\ex[Y\mid\mathcal F_{t\wedge\tau}]
\Bigr]
\ge
\ex[X]\ex[Y]
\end{equation}
Assume for the moment that \eqref{EqHDR} holds for every $t$. Since the filtration is constant after time $|E|$, we have $\mathcal F_{\tau}=\mathcal F_{|E|\wedge\tau}$, and therefore
\[
\ex\Bigl[
\ex[X\mid\mathcal F_\tau]
\,
\ex[Y\mid\mathcal F_\tau]
\Bigr]
\ge
\ex[X]\ex[Y].
\]
By Remark~\ref{rem:FKG1},
\[
\ex\Bigl[
\ex[X\mid\mathcal F_\tau]
\,
\ex[Y\mid\mathcal F_\tau]
\Bigr]
=
\ex\Bigl[
X\,\ex[Y\mid\mathcal F_\tau]
\Bigr],
\]
which yields the desired conclusion. 

The case $t=0$ is immediate since
$\mathcal F_0$ is the trivial $\sigma$-algebra. Now fix $t\in\mathbb N$ and assume that \eqref{EqHDR} holds at time $t$. It is enough to prove that, for every possible realization $\eta$ of $\pi_{t\wedge\tau}$ (that is, every $\eta$ such that $\pr [\pi_{t\wedge\tau}=\eta]>0$), we have
\begin{equation}
\label{E:proofFKG0}
\ex\Bigl[
\ex[X\mid\mathcal F_{(t+1)\wedge\tau}]
\,
\ex[Y\mid\mathcal F_{(t+1)\wedge\tau}]
\Bigm|
\pi_{t\wedge\tau}=\eta
\Bigr]
\ge
\ex\Bigl[
\ex[X\mid\mathcal F_{t\wedge\tau}]
\,
\ex[Y\mid\mathcal F_{t\wedge\tau}]
\Bigm|
\pi_{t\wedge\tau}=\eta
\Bigr].
\end{equation}
Indeed, integrating over $\eta$ and using the induction hypothesis
immediately yields \eqref{EqHDR} at time $t+1$.

Fix such a realization $\eta$. If, on the event $\{\pi_{t\wedge\tau}=\eta\}$, either $\tau\le t$ or
the cluster of $\ori$ has already been completely explored by time
$t\wedge\tau$, then no new edge is revealed between times
$t\wedge\tau$ and $(t+1)\wedge\tau$. Hence $\pi_{(t+1)\wedge\tau}=\eta$, and therefore
\begin{equation*}
\ex[X\mid\pi_{(t+1)\wedge\tau}=\eta]
=
\ex[X\mid\pi_{t\wedge\tau}=\eta],
\qquad
\ex[Y\mid\pi_{(t+1)\wedge\tau}=\eta]
=
\ex[Y\mid\pi_{t\wedge\tau}=\eta].
\end{equation*}
Thus \eqref{E:proofFKG0} holds with equality.

Assume now that this is not the case. Then, on the event $\{\pi_{t\wedge\tau}=\eta\}$, $\tau>t$, and the cluster of $\ori$ has not yet been completely explored by time $t$. Let $e$ denote the (deterministic) edge explored at time $t+1$ on the event $\{\pi_{t\wedge\tau}=\eta\}$. Let $\omega$ be such that $\omega_{t\wedge\tau}=\eta$. Then
\[
e\in \mathcal{H}_{(t+1)\wedge\tau}(\omega)
\subseteq \mathcal{H}_\tau(\omega).
\]
Since both $X$ and $Y$ are increasing up to $\tau$, we have
\[
X\bigl(\omega^{(e\to0)}\bigr)
\le
X\bigl(\omega^{(e\to1)}\bigr),
\qquad
Y\bigl(\omega^{(e\to0)}\bigr)
\le
Y\bigl(\omega^{(e\to1)}\bigr).
\]
Averaging these inequalities conditionally on $\{\pi_{t\wedge\tau}=\eta\}$ gives
\[
\ex\!\left[
X\bigl(\pi^{(e\to0)}\bigr)
\,\middle|\,
\pi_{t\wedge\tau}=\eta
\right]
\le
\ex\!\left[
X\bigl(\pi^{(e\to1)}\bigr)
\,\middle|\,
\pi_{t\wedge\tau}=\eta
\right],
\]
and similarly for $Y$. Since $e\notin\mathcal H_{t\wedge\tau}$ on the event
$\{\pi_{t\wedge\tau}=\eta\}$, the product structure implies that, for $i\in\{0,1\}$,
\[
\ex\!\left[
X\bigl(\pi^{(e\to i)}\bigr)
\,\middle|\,
\pi_{t\wedge\tau}=\eta
\right]
=
\ex\!\left[
X
\,\middle|\,
\pi_{t\wedge\tau}=\eta,\ \pi_e=i
\right].
\]
Therefore,
\[
\ex[X\mid\pi_{t\wedge\tau}=\eta,\pi_e=0]
\le
\ex[X\mid\pi_{t\wedge\tau}=\eta,\pi_e=1],
\]
and similarly for $Y$. Introduce the notation
\[
X_\eta^i
=
\ex[X\mid\pi_{t\wedge\tau}=\eta,\pi_e=i],
\qquad
Y_\eta^i
=
\ex[Y\mid\pi_{t\wedge\tau}=\eta,\pi_e=i],
\]
for $i=0,1$. Then
\begin{equation}
\label{E:proofFKG2bis}
X_\eta^0\le X_\eta^1,
\qquad
Y_\eta^0\le Y_\eta^1.
\end{equation}

Since the edge $e$ is determined by $\eta$ but has not yet been
explored, independence implies that its conditional state given
$\{\pi_{t\wedge\tau}=\eta\}$ is Bernoulli with parameter $p$.

Thus, on the one hand,
\begin{align}
\label{E:proofFKG3}
&\ex\Bigl[
\ex[X\mid\mathcal F_{(t+1)\wedge\tau}]
\,
\ex[Y\mid\mathcal F_{(t+1)\wedge\tau}]
\Bigm|
\pi_{t\wedge\tau}=\eta
\Bigr]
\nonumber\\
&=(1-p)
\ex\Bigl[
\ex[X\mid\mathcal F_{(t+1)\wedge\tau}]
\,
\ex[Y\mid\mathcal F_{(t+1)\wedge\tau}]
\Bigm|
\pi_{t\wedge\tau}=\eta,\pi_e=0
\Bigr]
\nonumber\\
&\qquad
+p
\ex\Bigl[
\ex[X\mid\mathcal F_{(t+1)\wedge\tau}]
\,
\ex[Y\mid\mathcal F_{(t+1)\wedge\tau}]
\Bigm|
\pi_{t\wedge\tau}=\eta,\pi_e=1
\Bigr]
\nonumber\\
&=(1-p)X_\eta^0Y_\eta^0+pX_\eta^1Y_\eta^1.
\end{align}
On the other hand,
\begin{align}
\label{E:proofFKG4}
&\ex\Bigl[
\ex[X\mid\mathcal F_{t\wedge\tau}]
\,
\ex[Y\mid\mathcal F_{t\wedge\tau}]
\Bigm|
\pi_{t\wedge\tau}=\eta
\Bigr]
\nonumber\\
&=
\ex[X\mid\pi_{t\wedge\tau}=\eta]\,
\ex[Y\mid\pi_{t\wedge\tau}=\eta]
\nonumber\\
&=
\Bigl[(1-p)X_\eta^0+pX_\eta^1\Bigr]
\Bigl[(1-p)Y_\eta^0+pY_\eta^1\Bigr].
\end{align}
Subtracting the right-hand sides of \eqref{E:proofFKG3} and \eqref{E:proofFKG4} gives
\begin{align*}
&(1-p)X_\eta^0Y_\eta^0+pX_\eta^1Y_\eta^1
\nonumber\\
&\qquad
-
\Bigl[(1-p)X_\eta^0+pX_\eta^1\Bigr]
\Bigl[(1-p)Y_\eta^0+pY_\eta^1\Bigr]
\nonumber\\
&=
p(1-p)
\Bigl[
X_\eta^0Y_\eta^0
+X_\eta^1Y_\eta^1
-X_\eta^0Y_\eta^1
-X_\eta^1Y_\eta^0
\Bigr]
\nonumber\\
&=
p(1-p)
\Bigl[
X_\eta^1-X_\eta^0
\Bigr]
\Bigl[
Y_\eta^1-Y_\eta^0
\Bigr]
\nonumber\\
&\ge0,
\qquad\text{by~\eqref{E:proofFKG2bis}}.
\end{align*}
Since this difference is nonnegative, comparing \eqref{E:proofFKG3} and \eqref{E:proofFKG4} yields \eqref{E:proofFKG0}. This completes the induction and the proof.
\end{proof}

%% file: Acknowledgements.tex
\paragraph*{Acknowledgements.}
The author would like to thank Raphaël Cerf for reading the manuscript
and for his valuable feedback.